\documentclass[11pt]{amsart}
\usepackage[utf8]{inputenc}
\usepackage{amsfonts,amssymb,amsmath,amsthm}
\usepackage{xcolor}
\usepackage{graphicx}
\usepackage{appendix}
\newtheorem{theorem}{Theorem}

\newtheorem{lemma}[theorem]{Lemma}

\newtheorem{problem}[theorem]{Problem}
\newtheorem{proposition}[theorem]{Proposition}

\newtheorem{conjecture}[theorem]{Conjecture}
\theoremstyle{definition}
\newtheorem{definition}[theorem]{Definition}
\newtheorem{remark}[theorem]{Remark}
\numberwithin{equation}{section}

\title{On the restricted isometry property of the Paley matrix}
\author{Shohei Satake}
\thanks{Faculty of Advanced Science and Technology, Kumamoto University, 2-39-1, Kurokami, Chuo, Kumamoto, Japan, 860-8555.
e-mail: shohei-satake@kumamoto-u.ac.jp}
\date{}



\begin{document}

\maketitle

\begin{abstract}
In this paper, we prove that the Paley graph conjecture implies that the Paley matrix has restricted isometry property (RIP) beating the square-root bottleneck for the sparsity level. 
Moreover, we show that the RIP of the Paley matrix implies an improved bound on the size of transitive subtournaments in the Paley tournament. 
\end{abstract}

\section{Introduction}
\label{sect-intro}
Matrices with {\it restricted isometry property} ({\it RIP}) defined below have important applications to signal processing since, by adopting them, it is possible to measure and recover sparse signals using significantly fewer measurements than the dimension of the signals~\cite{Ca2008}. 

\begin{definition}[Restricted isometry property, RIP]
Let $\Phi$ be a complex $M \times N$ matrix. Suppose that $K \leq M \leq N$ and $0 \leq \delta<1$. 
Then $\Phi$ is said to have the {\it $(K, \delta)$-restricted isometry property} ({\it RIP}) if 
\begin{equation}
    (1-\delta)||\mathbf{x}||^2 \leq ||\Phi \mathbf{x}||^2 \leq  (1+\delta)||\mathbf{x}||^2
\end{equation}
for every $N$-dimensional complex vector $\mathbf{x}$ with at most $K$ non-zero entries. Here $||\cdot||$ denotes the $\ell_2$ norm.
\end{definition}
According to Cand\`{e}s~\cite{Ca2008}, for applications to signal processing, it suffices to investigate the $(K, \delta)$-RIP matrix for some $\delta<\sqrt{2}-1$. 
In addition, the {\it sparsity} $K$ is expected to be as large as possible.

On the other hand, it is known that the problem checking whether a given matrix has RIP is NP-hard~\cite{BDMS2013}.
Thus many publications have attempted to give deterministic constructions of matrices having RIP.  

Throughout this paper, we assume that all matrices have column vectors with unit $\ell_2$-norm.
Most of known constructions use the {\it coherence} $\mu(\Phi)$ of an $M \times N$ matrix $\Phi$ with column vectors $\psi_1, \ldots, \psi_N$, where 
\begin{equation}
    \mu(\Phi):=\max_{1 \leq j\neq k \leq N}|\langle \psi_j, \psi_k \rangle|,
\end{equation}
and $\langle \cdot, \cdot \rangle$ denotes the standard inner product in the Hilbert space $\mathbb{C}^M$.
It can be proved (e.g. \cite{BDFKK2011}) that if $\mu(\Phi)=\mu$, then for every $K\leq M$, $\Phi$ has the $(K, (K-1)\mu)$-RIP, which implies the $(K, \delta)$-RIP with only $K=O(\sqrt{M})$, following from the below Welch bound (\ref{eq-Welch}) in \cite{W1974}. 
\begin{equation}
\label{eq-Welch}
    \mu(\Phi) \geq \sqrt{\frac{N-M}{M(N-1)}}.
\end{equation}
This barrier for the magnitude of the order of $K$ is called the {\it square-root bottleneck} or {\it quadratic bottleneck}.
From this situation, the following problem arises.

\begin{problem}[\cite{BDFKK2011}]
\label{prob-beatsqrt}
Construct an $M \times N$ matrix $\Phi$ having the $(K, \delta)$-RIP with $K=\Omega(M^{\gamma})$ for some $\gamma>1/2$ and $\delta<\sqrt{2}-1$.
\end{problem}
To our best knowledge, the first (unconditional) solution to this problem was given by Bourgain, Dilworth, Ford, Konyagin and Kutzarova~\cite{BDFKK2011}, and later was generalized by Mixon~\cite{M2015}.
It has been conjectured (\cite{BFMW2013}) that the {\it Paley matrix}, a $(p+1)/2 \times (p+1)$ matrix defined by quadratic residues modulo an odd prime $p$ (see Section~\ref{sect-prelim}), satisfies the $(K, \delta)$-RIP with $K\geq C_{\delta}\cdot p/{\rm polylog}\:p$ for some $C_{\delta}>0$ depending only on $\delta$;  
while the authors of \cite{BFMW2013} focused on primes $p\equiv 1 \pmod{4}$, as far as we know, there is no facts or evidences which prevent us from expecting that the conjecture also holds for primes $p\equiv 3 \pmod{4}$.
Under a number-theoretic conjecture shown in Section~\ref{sect-prelim}, 
Bandeira, Mixon and Moreira~\cite{BMM2017} proved that when $p \equiv 1 \pmod{4}$, the Paley matrix has the $(\Omega(p^{\gamma}), o(1))$-RIP for some $\gamma>1/2$, which provides a conditional solution to Problem~\ref{prob-beatsqrt}.

The main results of this paper are summarized as follows.
First, assuming that the widely-believed {\it Paley graph conjecture} formulated in Section~\ref{sect-prelim} holds, we prove that the Paley matrix is a solution to Problem~\ref{prob-beatsqrt} for any sufficiently large prime $p$, including primes $p \equiv 3 \pmod{4}$ beyond the scope in \cite{BMM2017}. 
Second, corresponding to a result in \cite{BMM2017} estimating the clique number of the Paley graph, 
we prove that the RIP of the Paley matrix implies a new upper bound on the size of transitive subtournaments (i.e. ones with no directed cycles) in the {\it Paley tournament} defined in Section~\ref{sect-main2}. The bound here is significantly better than the existing bounds by Tabib~\cite{T1986}, Momihara and Suda~\cite{MS2017} for this tournament.

The remainder of this paper is organized as follows.
Section~\ref{sect-prelim} introduces the Paley matrix and Paley graph conjecture, together with some key notions related to RIP.
Sections~\ref{sect-main1} and \ref{sect-main2} prove the main results.
As a byproduct, Appendix provides a new unconditional upper bound on the size of transitive subtournaments in the Paley tournament using a recent result by Hanson and Petridis~\cite{HP2020}.

\section{Preliminaries}
\label{sect-prelim}
Throughout this paper, let $p$ denote an odd prime number.
Let $\mathbb{F}_p$ be a finite field with $p$ elements which can be identified to the residue ring $\mathbb{Z}/p\mathbb{Z}$.
It is well-known that the multiplicative group of $\mathbb{F}_p$, denoted by $\mathbb{F}_p^*$, is a cyclic group of order $p-1$, consisting of all non-zero elements of $\mathbb{F}_p$.
A non-zero element $a \in \mathbb{F}_p$ is called a {\it quadratic residue modulo $p$} if the equation $X^2 \equiv a \pmod{p}$ has non-zero solutions.
Note that there exist exactly $(p-1)/2$ quadratic residues modulo $p$.
The {\it canonical additive character} $\psi$ of $\mathbb{F}_p$ is the map from $\mathbb{F}_p$ to the unit circle in $\mathbb{C}$ such that $\psi(x):=\exp(\frac{2 \pi \sqrt{-1}}{p} \cdot x)$ for all $x \in \mathbb{F}_p$.
Notice that for every pair of $x, y \in \mathbb{F}_p$, we have $\psi(x+y)=\psi(x)\psi(y)$.
A {\it quadratic multiplicative character} $\chi$ of $\mathbb{F}_p$ is a map from $\mathbb{F}_p$ to $\{0, \pm 1\}$ defined as
\begin{equation}
\chi(x):=
    \begin{cases}
    0  & x=0;\\
    1  & \text{$x$ is a quadratic residue modulo $p$};\\
    -1 & \text{otherwise}.
    \end{cases}
\end{equation}
Notice that $\chi(xy)=\chi(x)\chi(y)$ for every pair of $x, y \in \mathbb{F}_p$.

We are ready to define the Paley matrix.

\begin{definition}[Paley matrix, \cite{BMM2017}, \cite{R2007}, \cite{Z1999}]
\label{def-paley-matrix}
Let $Q_p$ denote the set of all quadratic residues modulo $p$; recall that $|Q_p|=(p-1)/2$.
Suppose that elements of $\mathbb{F}_p$ and $Q_p$ are labelled as $\mathbb{F}_p=\{0=a_1, a_2, \ldots, a_p\}$ and $Q_p=\{b_1, b_2, \ldots, b_{(p-1)/2}\}$, respectively.
Define $r$ as $0$ if $p \equiv 1 \pmod{4}$ and $1$ if $p \equiv 3 \pmod{4}$.

Then the {\it Paley matrix} $\Phi_p$ is a $(p+1)/2 \times (p+1)$ complex matrix of the following form.
\begin{equation}
  \Phi_p := \left[
    \begin{array}{ccccc}
      \frac{1}{\sqrt{p}} & \frac{1}{\sqrt{p}} & \ldots & \frac{1}{\sqrt{p}} & (\sqrt{-1})^r \\ \\
      \sqrt{\frac{2}{p}} & \sqrt{\frac{2}{p}}\psi(b_1a_{2}) & \ldots & \sqrt{\frac{2}{p}}\psi(b_1a_{p}) & 0 \\ \\
      \sqrt{\frac{2}{p}} & \sqrt{\frac{2}{p}}\psi(b_2a_{2}) & \ldots & \sqrt{\frac{2}{p}}\psi(b_2a_{p}) & 0 \\ \\
      \vdots & \vdots & \ddots & \vdots & \vdots \\ \\
      \sqrt{\frac{2}{p}} & \sqrt{\frac{2}{p}}\psi \bigl(b_{\frac{p-1}{2}}a_{2} \bigr) & \ldots & \sqrt{\frac{2}{p}}\psi \bigl(b_{\frac{p-1}{2}}a_{p} \bigr) & 0
    \end{array}
  \right]
\end{equation}
\end{definition}
It is not difficult to check that each column of $\Phi_p$ has $\ell_2$-norm $1$.
Note that column vectors of $\Phi_p$ produce an equiangular tight frame (\cite{R2007}), implying that $\Phi_p$ has the optimal coherence with respect to the Welch bound (\ref{eq-Welch}).

The following {\it flat restricted isometry property} ({\it flat RIP}) plays a key role to derive the RIP of the Paley matrix $\Phi_p$.

\begin{definition}[Flat RIP, e.g. \cite{BDFKK2011}, \cite{BFMW2013}]
Let $\Phi$ be an $M \times N$ matrix with columns $\psi_1, \ldots, \psi_N$. Suppose that $K \leq M \leq N$ and $\theta>0$. 
Then $\Phi$ is said to have the {\it $(K, \theta)$-flat restricted isometry property} ({\it flat RIP}) if 
\begin{equation}
   \Bigl |\Bigl\langle \sum_{i \in I}\psi_i, \sum_{j\in J}\psi_j \Bigr\rangle \Bigr|\leq \theta \sqrt{|I||J|}
\end{equation}
for every pair of disjoint subsets $I, J \subset \{1, 2, \ldots, N\}$ with $|I|, |J| \leq K$.
\end{definition}

\begin{proposition}[e.g. \cite{BDFKK2011}, \cite{BFMW2013}]
\label{prop-flatRIP}
A matrix $\Phi$ has the $(K, 150\: \theta \log K)$-RIP provided that $\Phi$ has the $(K, \theta)$-flat RIP.
\end{proposition}

In \cite{BMM2017}, the authors proved that in the case of $p \equiv 1 \pmod{4}$, 
the flat-RIP holds for the Paley matrix $\Phi_p$ by assuming that the following Conjecture~\ref{conj-BMM2017} is true, which induces a conditional solution to Problem~\ref{prob-beatsqrt}. 

\begin{conjecture}[\cite{C1994}]
\label{conj-BMM2017}
Let $0<\alpha \leq 1$ be a real number.
Then there exist $\beta=\beta(\alpha)>0$ and $p(\alpha)>0$ such that for any prime $p>p(\alpha)$ and any subset $S \subset \mathbb{F}_p$ with $|S|>p^{\alpha}$,
$$
\Bigl|\sum_{s_1, s_2 \in S}\chi(s_1-s_2) \Bigr| \leq |S|^{2-\beta}
$$
\end{conjecture}

In this paper, we aim to extend the results in~\cite{BMM2017} to general primes, including primes $p \equiv 3 \pmod{4}$.
While Conjecture~\ref{conj-BMM2017} is non-trivial for the case that $p \equiv 1 \pmod{4}$, it is trivial when $p \equiv 3 \pmod{4}$ since if $p \equiv 3 \pmod{4}$, it holds for any subset $S \subset \mathbb{F}_p$ that $\sum_{s_1, s_2 \in S}\chi(s_1-s_2)=0$, which follows from a simple fact that $\chi(-1)=-1$ when $p \equiv 3 \pmod{4}$.
Thus Conjecture~\ref{conj-BMM2017} is not enough to verify the desired RIP of the Paley matrix $\Phi_p$ for the case of $p \equiv 3 \pmod{4}$.
To deal with both of primes $p \equiv 1 \pmod{4}$ and $p \equiv 3 \pmod{4}$, we will make use of the following well-known {\it Paley graph conjecture}; see e.g. \cite{C2008}, \cite{CG1988}, \cite{GM2020}.

\begin{conjecture}[Paley graph conjecture]
\label{conj-pgc}
Let $p$ be a prime.
For $0<\alpha \leq 1$ and $\beta>0$, we say that the property $\mathcal{P}(\alpha, \beta)$ holds if for every pair of $S, T \subset \mathbb{F}_p$ with $|S|, |T|>p^{\alpha}$,
\begin{equation}
\label{eq-pgc}
   \Bigl |\sum_{s \in S, t \in T}\chi(s-t) \Bigr|\leq p^{-\beta}|S||T|.
\end{equation}
Then for each $0<\alpha \leq 1$, there exist $\beta=\beta(\alpha)>0$ and $p(\alpha)>0$ such that $\mathcal{P}(\alpha, \beta)$ holds for any prime $p>p(\alpha)$. 
\end{conjecture}

\begin{remark}
\label{rmk-beta}
It is known (\cite{CG1988}) that $\beta=\beta(\alpha)<1/2$ for every $0<\alpha<1$.
\end{remark}

\begin{remark}
\label{rmk-clique}
Let us introduce a connection between Conjecture~\ref{conj-pgc} and the Paley graph. 
For a prime $p \equiv 1 \pmod{4}$, the {\it Paley graph} $G_p$ with $p$ vertices is defined as an undirected graph with vertex set $\mathbb{F}_p$ in which two distinct vertices $x$ and $y$ are adjacent if and only if $\chi(x-y)=1$.
Note that this is well-defined since $\chi(x-y)=\chi(y-x)$ for any $x, y \in \mathbb{F}_p$ by the assumption of $p$.
The clique number (i.e. the size of the maximum cliques), denoted by $\omega(G_p)$, of $G_p$ has been extensively studied in graph theory and additive combinatorics. 
The best known upper bound was given in a recent paper by Hanson and Petridis~\cite{HP2020} proving that $\omega(G_p)\leq \sqrt{p/2}+1$. 
Conjecture~\ref{conj-pgc} implies a stronger bound that $\omega(G_p)\leq p^{\varepsilon}$ for any $\varepsilon>0$ and sufficiently large $p$.
In fact, if $S \subset \mathbb{F}_p$ induces a clique of $G_p$, then $|\sum_{s_1, s_2 \in S}\chi(s_1-s_2)|=|S|(|S|-1)$, which contradicts the claim of Conjecture~\ref{conj-pgc} if $|S|>p^{\varepsilon}$.
\end{remark}

\section{Verifying the RIP of the Paley matrix}
\label{sect-main1}

This section proves the following theorem which is the first main result in this paper. 
\begin{theorem}
\label{thm-main1}
Suppose that Conjecture~\ref{conj-pgc} holds.
Let $0<\alpha<1/2$ be a real number and $p$ a prime with $p>p(\alpha)$; $p(\alpha)$ is from Conjecture~\ref{conj-pgc}.
Then there exists some $\beta_0=\beta_0(\alpha)>0$ with $\alpha+\beta_0<1/2$ such that
for any $\tau$ with 
$$
\max \Bigl\{\alpha+\beta_0,\: \frac{1}{2}-\beta_0 \Bigr\}<\tau<\frac{1}{2},
$$ 
the Paley matrix $\Phi_p$ has the $(p^{\tau+\beta_0}, p^{\tau-1/2+o(1)})$-RIP.
In particular, $\Phi_p$ has the $(\Omega(p^{\gamma}),o(1))$-RIP for some $\gamma>1/2$.
\end{theorem}

Thus under Conjecture~\ref{conj-pgc}, Theorem~\ref{thm-main1} provides a solution to Problem~\ref{prob-beatsqrt}; since $\Phi_p$ is a $(p+1)/2 \times (p+1)$ matrix, Theorem~\ref{thm-main1} implies that $\Phi_p$ has the $(K, o(1))$-RIP with $K=\Omega(p^{\gamma})=\Omega(M^{\gamma})$ for some $\gamma>1/2$. 

Before the proof of Theorem~\ref{thm-main1}, we prove some key lemmas.
By Proposition~\ref{prop-flatRIP}, it suffices to verify the flat RIP of the Paley matrix $\Phi_p$, and so, it is necessary to compute the inner products of two distinct column vectors. 
To do this, we shall use the following well-known lemma on quadratic Gauss sums.

\begin{lemma}[e.g. \cite{LN1994}]
\label{lem-Gauss}
For $a \in \mathbb{F}_p^*$, 
\begin{equation}
\sum_{x \in \mathbb{F}_p}\psi(ax^2)=(\sqrt{-1})^r\chi(a)\sqrt{p},
\end{equation}
where $r$ was defined in Definition~\ref{def-paley-matrix}.
\end{lemma}

\begin{lemma}
\label{lem-inner}
Let $\phi_i$ be the $i$-th column of $\Phi_p$. Then, for each $1 \leq i\neq j \leq p$, 
\begin{equation}
    \langle \phi_i, \phi_j \rangle=\frac{(\sqrt{-1})^r}{\sqrt{p}} \cdot \chi(a_i-a_j).
\end{equation}
\end{lemma}

\begin{proof}
Note that for every $1 \leq k \leq \frac{p-1}{2}$, the equation $X^2 \equiv b_k \pmod{p}$ has exactly two distinct non-zero solutions by the definition of $b_k$.
Then it follows from Lemma~\ref{lem-Gauss} that
\begin{align}
\begin{split}
    \langle \phi_i, \phi_j \rangle
    &=\frac{1}{p}+\frac{2}{p}\sum_{k=1}^{\frac{p-1}{2}}\psi\bigl((a_i-a_j)b_k \bigr)\\
    &=\frac{1}{p}+\frac{2}{p}\cdot \frac{1}{2}\sum_{x \in \mathbb{F}_p^*}\psi\bigl((a_i-a_j)x^2 \bigr)\\
    &=\frac{1}{p}\sum_{x \in \mathbb{F}_p}\psi\bigl((a_i-a_j)x^2 \bigr)\\
    &=\frac{1}{p} \cdot (\sqrt{-1})^r\chi(a_i-a_j)\sqrt{p}=\frac{(\sqrt{-1})^r}{\sqrt{p}} \cdot \chi(a_i-a_j).
\end{split}
\end{align}
\end{proof}

The following lemma follows from Conjecture~\ref{conj-pgc}, which is another key tool to verify the flat RIP of the Paley matrix $\Phi_p$.

\begin{lemma}
\label{lem-doubsum}
Let $0<\alpha<1/2$. 
Suppose that there exists $\beta=\beta(\alpha)>0$ with $\alpha+\beta<1/2$ such that the property $\mathcal{P}(\alpha, \beta)$ holds.
Let $\tau$ be any real number such that $\alpha+\beta<\tau < 1/2$. 
Then it holds that
\begin{equation}
\label{eq-flatpgc}
     \Bigl |\sum_{s \in S, t \in T}\chi(s-t) \Bigr| \leq p^{\tau}\sqrt{|S||T|}
\end{equation}
for every pair of $S, T \subset \mathbb{F}_p$ with $|S|, |T| \leq p^{\tau+\beta}$.
\end{lemma}

\begin{proof}
Let $S, T \subset \mathbb{F}_p$ with $|S|, |T| \leq p^{\tau+\beta}$.
The proof is done by considering the following cases.

{\it Case 1.} 
If $|S||T|\leq p^{2\tau}$, then, by the trivial bound of $|\sum_{s \in S, t \in T}\chi(s-t)|$, we have
\begin{equation*}
    \Bigl |\sum_{s \in S, t \in T}\chi(s-t) \Bigr|
    \leq |S||T|
    =\sqrt{|S||T|}\cdot \sqrt{|S||T|} 
    \leq p^{\tau}\sqrt{|S||T|}.
\end{equation*}

{\it Case 2.} 
Next, suppose that $|S||T|> p^{2\tau}$ and we may assume $|S|>p^{\tau}$ without loss of generality.

{\it Case 2.1.} 
If $|T| \leq p^{\alpha}$, then the following inequalities hold by the assumption that $|S|\leq p^{\tau+\beta}$ $\colon$
\begin{align*}
    \Bigl |\sum_{s \in S, t \in T}\chi(s-t) \Bigr|
    \leq |S||T|
    &=\sqrt{|S||T|}\cdot \sqrt{|S||T|}\\ 
    & \leq \sqrt{p^{\alpha}\cdot p^{\tau+\beta}}\sqrt{|S||T|}
    = p^{\frac{\tau+\alpha+\beta}{2}} \sqrt{|S||T|}<p^{\tau}\sqrt{|S||T|},
\end{align*}
where the last inequality follows from the assumption that $\tau>\alpha+\beta$.

{\it Case 2.2.} 
If $|T|>p^{\alpha}$, since $\tau>\alpha+\beta>\alpha$, the inequality (\ref{eq-pgc}) holds for $S$ and $T$ by the property $\mathcal{P(\alpha, \beta)}$. 
Thus by the assumption that $|S|, |T|\leq p^{\tau+\beta}$, we have
\begin{align*}
     \Bigl |\sum_{s \in S, t \in T}\chi(s-t) \Bigr|
    \leq p^{-\beta}|S||T|
    &=p^{-\beta} \sqrt{|S||T|}\cdot \sqrt{|S||T|} \\
    & \leq p^{-\beta} \cdot \sqrt{p^{2(\tau+\beta)}} \cdot \sqrt{|S||T|}
    =p^{\tau}\sqrt{|S||T|}.
\end{align*}
\end{proof}

Now we are ready to prove Theorem~\ref{thm-main1}.

\begin{proof}[Proof of Theorem~\ref{thm-main1}]
Let $0<\alpha<1/2$ be a real number and $p$ a prime greater than $p(\alpha)$. 
Then Conjecture~\ref{conj-pgc} implies that there exists $\beta=\beta(\alpha)>0$ such that the property $\mathcal{P}(\alpha, \beta)$ holds.
If $\alpha+\beta<1/2$, we may take $\beta_0=\beta$.
If $\alpha+\beta \geq 1/2$, choose $\beta_0<\beta$ so that $\alpha+\beta_0<1/2$; note that for every $\beta'$ with $0<\beta'\leq \beta$, the property $\mathcal{P}(\alpha, \beta)$ implies the weaker property $\mathcal{P}(\alpha, \beta')$.

Now choose any $\tau$ with $\max\{\alpha+\beta_0, 1/2-\beta_0\}<\tau<1/2$.
It suffices to prove that the Paley matrix $\Phi_p$ has the $(p^{\tau+\beta_0}, 2p^{\tau-1/2})$-flat RIP since this claim, together with Proposition~\ref{prop-flatRIP}, implies that $\Phi_p$ has the $(K, \delta)$-RIP with $K=p^{\tau+\beta_0}$ and $\delta=150 \cdot 2p^{\tau-1/2} \cdot \log p^{\tau+\beta_0} \leq p^{\tau-1/2+o(1)}$, which proves the theorem.

Recall that $\phi_i$ is the $i$-th column of $\Phi_p$. 
For every pair of disjoint subsets $I, J \subset \{1, \ldots, p\}$, we have by Lemma~\ref{lem-inner} that
\begin{equation}
\label{eq-main1}
   \Bigl |\Bigl\langle \sum_{i \in I}\phi_i, \sum_{j\in J}\phi_j \Bigr\rangle \Bigr|
   =\frac{1}{\sqrt{p}} \cdot \biggl|\sum_{i \in I, j\in J}\chi(a_i-a_j) \biggr|.
\end{equation}
Thus if $|I|, |J| \leq p^{\tau+\beta_0}$ and $p+1\notin I\cup J$, Lemma~\ref{lem-doubsum} implies that
\begin{equation}
\label{eq-main2}
    \Bigl |\Bigl\langle \sum_{i \in I}\phi_i, \sum_{j\in J}\phi_j \Bigr\rangle \Bigr|
   \leq \frac{1}{\sqrt{p}} \cdot p^{\tau}\sqrt{|I||J|}
   =p^{\tau-\frac{1}{2}}\sqrt{|I||J|}.
\end{equation}

Now let us consider the case that $p+1$ is contained in $I$ or $J$.
Note that $p+1$ can be contained in only one of $I$ and $J$. 
We may assume that $p+1 \in J$. 
Then we have by the triangle inequality that
\begin{equation}
\label{eq-main3}
    \Bigl |\Bigl\langle \sum_{i \in I}\phi_i, \sum_{j\in J}\phi_j \Bigr\rangle \Bigr|
   \leq
   \Bigl |\Bigl\langle \sum_{i \in I}\phi_i, \sum_{j\in J\setminus \{p+1\}}\phi_j \Bigr\rangle \Bigr|
   +\Bigl |\Bigl\langle \sum_{i \in I}\phi_i, \phi_{p+1} \Bigr\rangle \Bigr|.
\end{equation}
Thus if $|I|, |J| \leq p^{\tau+\beta_0}$ (which also implies that $|J\setminus \{p+1\}| \leq p^{\tau+\beta_0}$), it holds by (\ref{eq-main2}) and (\ref{eq-main3}) that
\begin{equation}
\label{eq-main4}
    \Bigl |\Bigl\langle \sum_{i \in I}\phi_i, \sum_{j\in J}\phi_j \Bigr\rangle \Bigr|
   \leq
   p^{\tau-\frac{1}{2}}\sqrt{|I||J|}+\frac{|I|}{\sqrt{p}}.
\end{equation}
Since $|I|=\sqrt{|I|} \cdot \sqrt{|I|}\leq p^{(\tau+\beta_0)/2}\sqrt{|I|} \leq p^{(\tau+\beta_0)/2}\sqrt{|I||J|}$, it holds by (\ref{eq-main4}) that 
\begin{equation}
\label{eq-main5}
    \Bigl |\Bigl\langle \sum_{i \in I}\phi_i, \sum_{j\in J}\phi_j \Bigr\rangle \Bigr|
   \leq
   p^{\tau-\frac{1}{2}}\sqrt{|I||J|}+p^{\frac{\tau+\beta_0}{2}-\frac{1}{2}}\sqrt{|I||J|}\leq  2p^{\tau-\frac{1}{2}}\sqrt{|I||J|}.
\end{equation}
Notice that $(\tau+\beta_0)/2<\tau$, equivalently, $\beta_0<\tau$ since $\beta_0< \alpha+\beta_0<\tau$.  

Thus by (\ref{eq-main2}) and (\ref{eq-main5}), the Paley matrix $\Phi_p$ has the $(p^{\tau+\beta_0}, 2p^{\tau-1/2})$-flat RIP.
\end{proof}

Note that Theorem~\ref{thm-main1} implies the following theorem on the clique number of the Paley graph $G_p$, which is essentially same as the result in \cite{BMM2017}

\begin{theorem}
\label{thm-BMM2017Paleygraph}
Let $p \equiv 1 \pmod{4}$ be a prime.
Suppose that the Paley matrix $\Phi_p$ has the $(p^{\tau+\beta_0}, p^{\tau-1/2+o(1)})$-RIP for some $0<\tau<1/2$ and $\beta_0>0$ such that $\tau+\beta_0>1/2$.
Then $\omega(G_p) \leq p^{\tau+o(1)}+1=o(\sqrt{p})$.
\end{theorem}

\section{The size of transitive subtournaments in the Paley tournament}
\label{sect-main2}
Let $p \equiv 3 \pmod{4}$ be a prime. 
The {\it Paley tournament} $T_p$ with $p$ vertices is an oriented complete graph with vertex set $\mathbb{F}_p$ in which there exists an arc from $x$ to $y$ if and only if $\chi(x-y)=1$.
Notice that this is well-defined since $\chi(y-x)=-\chi(x-y)$ for any $x, y \in \mathbb{F}_p$ by the assumption of $p$.
The following is the second main result which establishes a connection between the RIP of the Paley matrix $\Phi_p$ and estimating the size of transitive subtournaments in the Paley tournament $T_p$, which has been well-studied in the context of the Erd\H{o}s-Moser problem on transitive subtournaments (\cite{EM1964}) and oriented Ramsey numbers (\cite{T2012}); see e.g. \cite{MS2017}, \cite{S1994} and \cite{S1998}.

\begin{theorem}
\label{thm-main2}
Let $p \equiv 3 \pmod{4}$ be a prime.
Suppose that the Paley matrix $\Phi_p$ has the $(p^{\tau+\beta_0}, p^{\tau-1/2+o(1)})$-RIP for some $0<\tau<1/2$ and $\beta_0>0$ with $\tau+\beta_0>1/2$.
Then the size of transitive subtournaments in the Paley tournament $T_p$ is at most 
$\sqrt{3}p^{\tau+o(1)}+1=o(\sqrt{p})$.
\end{theorem}

It is worth noting that our theorem, together with Theorem~\ref{thm-main1}, gives not only an analogous result of Theorem~\ref{thm-BMM2017Paleygraph} in Section~\ref{sect-main1} but also a conditional answer to a problem by Momihara and Suda~\cite[p.242]{MS2017}.
Also note that Tabib~\cite{T1986} proved that the size of transitive subtournaments in the Paley tournament $T_p$ is at most $-3/2+\sqrt{3p+13/4}$; some improvements of the constant term were given in \cite{MS2017}.
It is readily seen that the bound $o(\sqrt{p})$ in Theorem~\ref{thm-main2} works significantly better than the bound from Tabib~\cite{T1986} when $p$ is sufficiently large.

\begin{proof}[Proof of Theorem~\ref{thm-main2}]
Let $U \subset \mathbb{F}_p$ be the vertex set of a transitive subtournament of $T_p$.
Let $u=|U|$. Now consider the $(p+1)/2 \times u$ submatrix $\Phi_U$ of $\Phi_p$ consisting of columns $\phi_i$ of $\Phi_p$ with $a_i \in U$.
Then, since the induced subtournament $T_p[U]$ is transitive, it holds by Lemma~\ref{lem-inner} that the $u \times u$ matrix $\Phi_U^T \overline{\Phi_U}$ satisfies that
\begin{equation}
    \Phi_U^T \overline{\Phi_U}=I_{u}+\frac{1}{\sqrt{p}}A
\end{equation}
where $I_u$ is the identity matrix of size $u$ and $A=(a_{j,k})$ denotes the Hermitian matrix of size $u$ defined as
\begin{equation}
    a_{j, k} := \begin{cases}
    \sqrt{-1} & j<k;\\
    0 & j=k;\\
    -\sqrt{-1} & j>k.
  \end{cases}
\end{equation}
Note that all eigenvalues of $\Phi_U^T \overline{\Phi_U}$ are real since $A$ (and so $\Phi_U^T \overline{\Phi_U}$) is Hermitian.
For a Hermitian matrix $B$, let $\lambda_{\max}(B)$ and $\lambda_{\min}(B)$ denote the maximum and minimum eigenvalue of $B$, respectively.
Then it suffices to prove either
\begin{equation}
\label{eq-max}
   \lambda_{\max}(\Phi_U^T \overline{\Phi_U}) \geq 1+\sqrt{\frac{u^2-1}{3p}},
\end{equation}
or 
\begin{equation}
\label{eq-min}
   \lambda_{\min}(\Phi_U^T \overline{\Phi_U}) \leq 1-\sqrt{\frac{u^2-1}{3p}}.
\end{equation}
In fact, all eigenvalues of $\Phi_U^T \overline{\Phi_U}$ are in the interval $[1-\delta, 1+\delta]$ with $\delta=p^{\tau-1/2+o(1)}$, which follows from the RIP of $\Phi_p$, together with the assumption that $K=p^{\tau+\beta_0} \gg \sqrt{p}$, and the fact that $u=O(\sqrt{p})$ as shown by Tabib~\cite{T1986}.
Thus (\ref{eq-max}) shows that
\begin{equation}
    1+\sqrt{\frac{u^2-1}{3p}} \leq 1+\delta.
\end{equation}
This implies that $u\leq \delta \sqrt{3p}+1$, which proves the theorem. 
Similarly, (\ref{eq-min}) shows that
\begin{equation}
    1-\sqrt{\frac{u^2-1}{3p}} \geq 1-\delta,
\end{equation}
which implies the same bound on $u$.

To evaluate $\lambda_{\max}(\Phi_U^T \overline{\Phi_U})$ or $\lambda_{\min}(\Phi_U^T \overline{\Phi_U})$, we shall focus on $\lambda_{\max}(A)$ or $\lambda_{\min}(A)$, respectively. 
Note that $\lambda_{\max}(A) \geq 0 \geq \lambda_{\min}(A)$ since the trace of $A$ is $0$. 
Here it is useful to consider $\lambda_{\max}(A^2)$. 
Notice that each eigenvalue of $A^2$ is a square of an eigenvalue of $A$. 

The proof is completed by considering the following two cases.

{\it Case 1.} 
First, if $|\lambda_{\max}(A)| \geq |\lambda_{\min}(A)|$,
then we have $\lambda_{\max}(A)=\sqrt{\lambda_{\max}(A^2)}$.
A direct calculation shows that $A^2=(b_{j,k})$ with 
\begin{equation}
    b_{j, k} := \begin{cases}
    u+2(j-k) & j<k;\\
    u-1 & j=k;\\
    u+2(k-j) & j>k.
  \end{cases}
\end{equation}
Thus it holds by a simple calculation that 
\begin{equation}
\label{eq-Rayleigh}
    \boldsymbol{1}_u^T A^2 \boldsymbol{1}_u=\frac{u(u^2-1)}{3},
\end{equation}
where $\boldsymbol{1}_u$ denotes the all-one vector of length $u$. 
Since $\langle \boldsymbol{1}_u, \boldsymbol{1}_u \rangle=u$, the mini-max theorem shows that
\begin{equation}
\label{eq-maxeigensqrt}
    \lambda_{\max}(A^2) \geq \frac{u^2-1}{3},
\end{equation}
and thus we have $\lambda_{\max}(A)\geq \sqrt{(u^2-1)/3}$, implying (\ref{eq-max}).

{\it Case 2.} 
If $|\lambda_{\max}(A)| < |\lambda_{\min}(A)|$, we may consider $\lambda_{\min}(\Phi_U^T \overline{\Phi_U})$.
In this case, $\lambda_{\min}(A)=-\sqrt{\lambda_{\max}(A^2)}$ holds.
By (\ref{eq-maxeigensqrt}), we have $\lambda_{\min}(A)\leq -\sqrt{(u^2-1)/3}$, implying (\ref{eq-min}).
\end{proof}

\begin{remark}
From a recent result by Hanson and Petridis~\cite{HP2020}, one can observe that the Tabib's bound can be 
improved to $1+\sqrt{2p-1}$, which outperforms the Tabib's bound when $p \geq 59$;
for details, see Appendix.
\end{remark}

\begin{remark}
As in the case of the Paley graph $G_p$, Conjecture~\ref{conj-pgc} implies that the size of transitive subtournaments in the Paley tournament $T_p$ is smaller than $p^{\varepsilon}$ for any $\varepsilon>0$ and any sufficiently large prime $p \equiv 3 \pmod{4}$; thus, Theorems~\ref{thm-main1} and \ref{thm-main2} show that the RIP of the Paley matrix $\Phi_p$ lies between Conjecture~\ref{conj-pgc} and estimating the size of transitive subtournaments in $T_p$.
To show this claim, suppose that $U \subset \mathbb{F}_p$ induces a transitive subtournament of $T_p$. 
Then there exists a linear order $<$ of vertices in $U$ such that $x<y$ if and only if $(x, y)$ is an arc of $T_p$ for any pair of $x, y \in U$.
Take a partition $(U_1, U_2)$ of $U$ such that $U_1$ is the set of the preceding $|U_1|$ vertices of $U$ with respect to the linear order $<$ and $U_2$ consists of all remained vertices of $U$.
Then we have
$$
\Bigl |\sum_{u_1 \in U_1, u_2 \in U_2}\chi(u_1-u_2) \Bigr|=|U_1||U_2|,
$$
which contradicts Conjecture~\ref{conj-pgc} if $|U|>p^\varepsilon$ for any $\varepsilon>0$ and any sufficiently large prime $p$.
\end{remark}

\section*{Acknowledgement}
The author thanks to Yujie Gu and Koji Momihara for their many constructive comments on earlier versions of this paper.
This research has been supported by Grant-in-Aid for JSPS Fellows  20J00469 of the Japan Society for the Promotion of Science.

\clearpage
\appendix
\section*{Appendix}
This section is to prove the following theorem.
\begin{theorem}
\label{thm-appendix}
Let $p \equiv 3 \pmod{4}$ be a prime.
Then the size of transitive subtournaments in the Paley tournament $T_p$ is at most 
$1+\sqrt{2p-1}$.
\end{theorem}

Theorem~\ref{thm-appendix} is a corollary of the following theorem by Hanson and Petridis~\cite{HP2020} on sumsets in the set $Q_p \cup \{0\}$.

\begin{theorem}[Theorem 1.2 in \cite{HP2020}]
\label{thm-HP2020}
Let $p$ be an odd prime and suppose that $A, B \subset \mathbb{F}_p$ satisfy $A+B \subset Q_p \cup \{0\}$.
Then it holds that
\begin{equation}
    |A||B| \leq \frac{p-1}{2}+|B \cap (-A)|.
\end{equation}
Here $A+B:=\{a+b \mid a \in A, b\in B\}$ and $-A:=\{-a \mid a \in A\}$.
\end{theorem}

\begin{proof}[Proof of Theorem~\ref{thm-appendix}]
Let $U \subset \mathbb{F}_p$ be the vertex set of a transitive subtournament of $T_p$.
Let $u=|U|$. 
Recall that there exists a linear order $<$ over $U$ such that for all $v_1, v_2 \in U$, $v_1<v_2$ if and only if $(v_1, v_2)$ is an edge of $T_p$.   
By this fact, for any integer $1\leq x \leq u-1$, there exists a partition $(U_1, U_2)$ of $U$ 
such that $|U_1|=x$ and $(u_1, u_2)$ is an edge of $T_p$ for all $u_1 \in U_1$, $u_2 \in U_2$.
Thus it holds by the definition of $T_p$ that $U_1-U_2=U_1+(-U_2) \subset Q_p \cup \{0\}$.
Since $|(-U_2) \cap (-U_1)| \leq |-U_1|=|U_1|=x$, it must hold by Theorem~\ref{thm-HP2020} that
\begin{equation}
\label{eq-appendix1}
    x(u-x)=|U_1||U_2| \leq  \frac{p-1}{2}+|(-U_2) \cap (-U_1)| \leq \frac{p-1}{2}+x
\end{equation}
for any integer $1\leq x \leq u-1$.
Equivalently, it must hold for any integer $1\leq x \leq u-1$ that
\begin{equation}
\label{eq-appendix2}
    \Bigl(x-\frac{u-1}{2} \Bigr)^2-\frac{(u-1)^2}{4}+\frac{p-1}{2}\geq 0.
\end{equation}
The theorem is obtained by a simple fact that the minimum value of the quadratic function in the left-hand side of (\ref{eq-appendix2}) among the integers in $[1, u-1]$ is $(u-1)^2/4+(p-1)/2$ if $u$ is odd and $(u-1)^2/4+(2p-1)/4$ if $u$ is even.
\end{proof}

\end{document}